\documentclass[11pt]{amsart}

\usepackage{ amscd,amsfonts,amsmath,amsthm,amssymb,mathrsfs,verbatim}
\usepackage{amsmath}
\usepackage[leqno]{amsmath}
\usepackage{mathtools}
\usepackage[table]{xcolor}

\newcommand{\ignore}[1]{}

\usepackage{graphicx}

\def\NZQ{\Bbb}               
\def\NN{{\NZQ N}}
\def\QQ{{\NZQ Q}}
\def\ZZ{{\NZQ Z}}

\def\frk{\frak}               

\def\Phi{{\frk n}}
\def\Phi{{\frk N}}
\def\MM{{\mathcal M}}

\def\MS{{\mathcal S}}

\def\MG{{\mathcal G}}

\def\ab{{\bold a}}

\def\yb{{\bold y}}
\def\zb{{\bold z}}
\def\ub{{\bold u}}
\def\vb{{\bold v}}
\def\tb{{\bold t}}
\def\wb{{\bold w}}

\def\cb{{\bold c}}

\def\td{{\rm td}}

\def\opn#1#2{\def#1{\operatorname{#2}}} 
\opn\chara{char} \opn\length{\ell} \opn\pd{pd} \opn\rk{rk}
\opn\projdim{proj\,dim} \opn\injdim{inj\,dim} \opn\rank{rank}
\opn\depth{depth} \opn\grade{grade} \opn\height{height}
\opn\embdim{emb\,dim} \opn\codim{codim}

\opn\Tr{Tr} \opn\bigrank{big\,rank}
\opn\superheight{superheight}\opn\lcm{lcm}
\opn\trdeg{tr\,deg}
\opn\reg{reg} \opn\lreg{lreg} \opn\ini{in} \opn\lpd{lpd}
\opn\size{size}\opn\bigsize{bigsize}
\opn\cosize{cosize}\opn\bigcosize{bigcosize}
\opn\sdepth{sdepth}\opn\sreg{sreg}
\opn\link{link}\opn\fdepth{fdepth}
\opn\div{div} \opn\Div{Div} \opn\cl{cl} \opn\Cl{Cl}
\opn\Spec{Spec} \opn\Supp{Supp} \opn\supp{supp} \opn\Sing{Sing}
\opn\Ass{Ass} \opn\Min{Min}\opn\Mon{Mon} \opn\dstab{dstab} \opn\astab{astab}
\opn\Ann{Ann} \opn\Rad{Rad} \opn\Soc{Soc}
\opn\Im{Im} \opn\Ker{Ker} \opn\Coker{Coker} \opn\Am{Am}
\opn\Hom{Hom} \opn\Tor{Tor} \opn\Ext{Ext} \opn\End{End}
\opn\Aut{Aut} \opn\id{id} \opn\span{span}

\opn\nat{nat}
\opn\pff{pf}
\opn\Pf{Pf} \opn\GL{GL} \opn\SL{SL} \opn\mod{mod} \opn\ord{ord}
\opn\Gin{Gin} \opn\Hilb{Hilb}\opn\sort{sort}
\opn\aff{aff} \opn\con{conv} \opn\relint{relint} \opn\st{st}
\opn\lk{lk} \opn\cn{cn} \opn\core{core} \opn\vol{vol}
\opn\link{link} \opn\star{star}\opn\lex{lex} \opn\Gr{Gr}
\opn\gr{gr}

\def\pot#1#2{#1[\kern-0.28ex[#2]\kern-0.28ex]}

\opn\dirlim{\underrightarrow{\lim}}
\opn\inivlim{\underleftarrow{\lim}}
\let\to=\rightarrow

\def\Implies{\ifmmode\Longrightarrow \else
        \unskip${}\Longrightarrow{}$\ignorespaces\fi}
\def\implies{\ifmmode\Rightarrow \else
        \unskip${}\Rightarrow{}$\ignorespaces\fi}
\def\iff{\ifmmode\Longleftrightarrow \else
        \unskip${}\Longleftrightarrow{}$\ignorespaces\fi}

\let\:=\colon
\newtheorem{Theorem}{Theorem}[section]
\newtheorem{Lemma}[Theorem]{Lemma}
\newtheorem{Corollary}[Theorem]{Corollary}
\newtheorem{Proposition}[Theorem]{Proposition}
\newtheorem{Remark}[Theorem]{Remark}

\newtheorem{Example}[Theorem]{Example}

\newtheorem{Question}[Theorem]{Question}

\let\epsilon\varepsilon
\let\kappa=\varkappa
\def\qed{\ifhmode\textqed\fi
      \ifmmode\ifinner\quad\qedsymbol\else\dispqed\fi\fi}
\def\textqed{\unskip\nobreak\penalty50
       \hskip2em\hbox{}\nobreak\hfil\qedsymbol
       \parfillskip=0pt \finalhyphendemerits=0}
\def\dispqed{\rlap{\qquad\qedsymbol}}

\opn\dis{dis}
\def\pnt{{\raise0.5mm\hbox{\large\bf.}}}

\opn\Lex{Lex}

\begin{document}

\title{Asymptotic  behavior of Markov complexity of matrices}



\author{Shmuel Onn, Apostolos Thoma  \and Marius Vladoiu}


\address{Shmuel Onn, Technion - Israel Institute of Technology, Haifa, Israel  }
\email{onn@technion.ac.il}
\address{Apostolos Thoma, Department of Mathematics, University of Ioannina, Ioannina 45110, Greece }
\email{athoma@uoi.gr}
\address{Marius Vladoiu, Faculty of Mathematics and Computer Science, University of Bucharest, Str. Academiei 14, Bucharest, RO-010014, Romania, and}
\address{Simion Stoilow Institute of Mathematics of Romanian Academy, Str. Grivita 21, Bucharest 014700, Romania}
\email{vladoiu@fmi.unibuc.ro}

\keywords{Toric ideals, Markov complexity, Graver basis, Lawrence liftings}
\subjclass{MSC 13P10, 05E40, 14M25, 15B36, 62R01}

\maketitle

\begin{abstract} { To any integer matrix $A$ one can associate a matroid
structure consisting of a graph and another integer matrix $A_B$. The connected components of this graph are called bouquets. We prove that bouquets behave well with respect to
the $r$--th  Lawrence liftings of matrices and we use it to prove that the
Markov and Graver complexities of $m\times n$ matrices of rank $d$
 may be arbitrarily large for $n\geq 4$ and $d\leq n-2$. 
In contrast, we show they are bounded in terms of $n$
 and the largest absolute value $a$ of any entry of $A$.\break
 }
\end{abstract}

\section{Introduction}
\label{intro}

Let $A\in{\mathbb Z}^{m\times n}$ be an integer matrix of rank $d$ and  $r\ge 2$. 
The $r$--th {\it Lawrence lifting} of $A$   is denoted by $A^{(r)}$ and is the $(rm+n)\times rn$ matrix
{\footnotesize\[
 A^{(r)}= \begin{array}{c} \overbrace{\quad\quad\quad \quad \quad\ \ }^{r-\textrm{times}} \\
 \begin{pmatrix}
\ A\ & 0 &  \hdots & 0 \\
0 &\ A\ & \hdots & 0 \\
\vdots & \vdots & \ddots & \vdots \\
0 & 0 & \hdots &\ A\ \\
I_n & I_n & \cdots & I_n
\end{pmatrix}\end{array}.
\]}

 We   identify an element of $ \Ker_{\ZZ}(A^{(r)})$ with  an $r\times n $ matrix:   each row of this matrix corresponds to an element of
$ \Ker_{\ZZ}(A)$ and the sum of its rows is zero. The   {\em type} of an element of $ \Ker_{\ZZ}(A^{(r)})$ is the number of  nonzero rows of this matrix.
The {\em Markov complexity}, $m(A)$, is the largest type of any vector in the universal Markov basis of $A^{(r)}$ as $r$ varies. The universal Markov basis
is the union of all minimal Markov bases, \cite{HS}.

Markov complexity measures the complexity of Markov bases for   hierarchical models in Algebraic Statistics. The finiteness of Markov
complexity was first observed by Aoki and Takemura, in \cite{AT},
  and the theory of Markov complexity was developed by Santos and Sturmfels in \cite{SS}.  Recent advances in understanding Markov complexity \cite{KT} and the development of the bouquets algebra \cite{PTV} permit us to answer the problem
of the maximal Markov complexity for any $m\times n$ integer matrix of rank $d$.

We will show that the maximal Markov complexity of  $m\times n$ integer matrices of  rank $d$ follows the following patern

\[
 \begin{array}{ccccccc}
0 & 2 & 3  & \infty & \infty & \cdots & \infty \\
 & 0 & 2  & \infty & \infty & \cdots & \infty \\
 &  & 0  & 2 & \infty & \cdots & \infty \\
 &  &   & 0 & 2 & \cdots & \infty \\
\end{array}
\]

where the columns correspond to $n\geq 1$, the rows to the rank $1\leq d\leq n$ and $\infty $ means that matrices of the corresponding dimensions  may have arbitrary large Markov complexity.
A specific matrix has always finite Markov complexity since the toric ideals of Lawrence liftings are
 positively graded, see \cite{HaraThomaVladoiu15}
and therefore the Markov complexity is smaller than or equal to the Graver complexity, which is always finite, see \cite{SS}.

Remark that if $d=n$ then the columns of $A$ are linearly independent and the kernel of $A$ and $A^{(r)}$, for every $r\ge 2$ are zero, therefore the Markov complexity of those matrices is zero.
In the case that $d=n-1$ the kernel of $A$ is generated by one vector
${\bf u}$ and the kernel of $A^{(r)}$ is generated by matrices of type 2, making the Markov complexity equal to 2.

The case of $1\times 3$ matrices was solved in \cite{HaraThomaVladoiu14}, where it was shown that the Markov complexity of an $1\times 3$ matrix is equal to
 two if the toric ideal $I_{A}$ is complete intersection and equal to three otherwise.
 Thus the maximal Markov complexity is three for $1\times 3$ matrices.
 Recently, in \cite{KT}, the case of $1\times n$, $n\geq 4$, was solved and it was proved that such $1\times n$ matrices may have arbitrarily large Markov complexity. In this article one main result is that $2\times 4$ matrices have arbitrarily large Markov complexity, see Theorem 5.1.
Then, \cite[Theorem 2.4]{KT}, which states that elimination behaves well with Markov complexity, provides a way to move horizontally from $(d,n)$-matrices that have arbitrarily large Markov complexity to prove that $(d,n+k)$-matrices have arbitrarily large Markov complexity. The bouquet's results of Section 3 move you diagonally from $(d,n)$-matrices that have arbitrarily large Markov complexity to prove that
$(d+k,n+k)$-matrices have arbitrarily large Markov complexity. The four results, \cite[Theorem 4.1]{KT}, Theorem~\ref{24},  \cite[Theorem 2.4]{KT} and Theorem~\ref{Markovcomplexity}, show that if we only bound the matrix dimensions then the Markov and Graver complexities can be arbitrarily large, the content of Theorem~\ref{final}. In contrast, based on recent results on sparse integer programming from \cite{EHKKLO,KLO,KO}, we show in Theorem~\ref{complexity_bound} that the Markov and Graver complexities are bounded in terms of $n$ and the largest absolute value $a$ of any entry of $A$.


\section{Preliminaries}

\label{section:preliminaries}


Let $\Bbbk$ be a field and $A=[\ab_1,\ldots,\ab_n]\in\ZZ^{m\times n}$ an integer matrix whose
column vectors are $\ab_1,\ldots,\ab_n$. We recall that the toric ideal of $A$ is the ideal
$I_A\subseteq \Bbbk[x_1,\ldots, x_n]$, generated by all binomials of the form $x^{\ub}-x^{\wb}$,
where $\ub-\wb\in\Ker_{\ZZ}(A)$. The columns of $A$ impose an $A$-grading
on $\Bbbk[x_1,\ldots, x_n]$ by defining the $A$-degree of $x_i$ to be $\ab_i $.
The toric ideal $I_A$ is $A$-homogeneous. Furthermore, if $\Ker_{\ZZ}(A)\cap\NN^n=\{{\bf 0}\}$ then this grading is  positive.

A \emph{Markov basis} $\MM$ of $A$ is a finite subset of $\Ker_{\ZZ}(A)$ such that whenever ${\bf w}, {\bf u} \in \mathbb{N}^n$ with  ${\bf w}-{\bf u} \in \Ker_{\ZZ}(A)$, there exists a subset $ \{ {\bf v}_i : i = 1, \ldots, p \}$ of $\MM$ that connects ${\bf w}$ to ${\bf u}$, that is $(\wb- {{\sum^l_{i=1}}}{\vb_i})\in \NN^n$  for all $1\leq l\leq p$ and   ${\bf w}- {\bf u}=\sum^p_{i=1}{\bf v}_i $. We call a Markov basis $\MM$ of $A$ {\it minimal} if no proper subset of $\MM$ is a Markov basis of $A$.
For a vector ${\bf u}\in  \Ker_{\ZZ}(A)$, we denote by $\bf u^+$, ${\bf u}^-$ the unique vectors in $ \NN^{n}$ such that  $\ub= \ub^+ -\ub^-$. According to a classical result by Diaconis and Sturmfels, if $\MM$  is a minimal Markov basis of $A$, then the set $\{x^{\bf u^+}-x^{\bf u^-}: \ {\bf u}\in \MM\}$ is a minimal generating set of $I_A$ (see \cite[Theorem 3.1]{DS}). The union of all minimal Markov bases of $A$, where we identify  elements that differ by a sign, is called the {\em universal Markov basis} of $A$ and is denoted by ${\MM}(A)$ (see \cite[Definition 3.1]{HS}).

The {\it indispensable subset} of the universal Markov basis $\MM(A)$, which is denoted by $\MS(A)$, is the intersection of all minimal Markov bases of $A$ via the same identification.
The {\it Graver basis} of $A$,  $\MG(A)$,  is the subset of  $\Ker_{\ZZ}(A)$ whose elements have {  no} {\it  proper conformal decomposition}; namely, an element ${\bf u}\in  \Ker_{\ZZ}(A)$ belongs to the Graver basis $\MG(A)$ if whenever ${\bf u}$ can be written in the form ${\bf v}+_{c} {\bf w}$,
where ${\bf v}, {\bf w} \in  \Ker_{\ZZ}(A)$ and  ${\bf u}^+={\bf v}^++ {\bf w}^+$,  ${\bf u}^-={\bf v}^-+ {\bf w}^-$,
we conclude that either ${\bf v}={\bf 0}$ or  ${\bf w}={\bf 0}$, see \cite[Chapter 4]{St}.
The {\it Graver basis} of $A$ is always a finite set and in the case that the ideal $I_A$ is homogeneous
for a positive grading it contains the universal Markov basis of $A$,  see \cite[Theorem 2.3]{HaraThomaVladoiu15}. The
ideal $I_A$ is homogeneous
for a positive grading if and only if
$\Ker_{\ZZ}(A)\cap\NN^n=\{{\bf 0}\}$.
Note that disregarding if the ideal $I_A$ is or is
not homogeneous for a positive grading the higher Lawrence liftings $A^{(r)}$ of $A$ always are,
since the row with all its entries 1 belongs to the row span of the  matrix $A^{(r)}$. Therefore, we have the following inclusions
 \[\MS(A^{(r)})\subseteq\MM(A^{(r)})\subseteq\MG(A^{(r)}) .\]

For the later proofs we will need the algebraic characterization of indispensable elements:

 \begin{Proposition}\cite[Proposition 1.1]{HaraThomaVladoiu14}
\label{indispensable}
 The set of indispensable elements ${\MS}(A)$ of $A$ consists  of all nonzero vectors in $\Ker_{\ZZ}({A})$ with no proper   semiconformal decomposition.
\end{Proposition}

We recall from \cite[Definition 3.9]{HS} that for vectors $\ub,\vb,\wb\in\Ker_{\ZZ}(A)$ such that $\ub=\vb+\wb$, the sum is said to be a semiconformal decomposition of $\ub$, written $\ub=\vb+_{sc} \wb$, if $v_i>0$ implies that $w_i\geq 0$, and $w_i<0$ implies that $v_i\leq 0$ for all $1\leq i\leq n$. Here $v_i$ denotes the $i^{th}$ coordinate of the vector $\vb$. As before, the decomposition is called {\it proper} if both $\vb, \wb$ are nonzero.
 We remark that  $\bf 0$ cannot be written as the semiconformal sum of two nonzero vectors  in the case that $ \Ker_{\ZZ}(A)\cap \NN^n=\{\bf 0\}$. Note also, that in a semiconformal sum $\ub=\vb+_{sc}\wb$ the following situations occur coordinate-wise: 1) if $u_i>0$ then $v_i$ can take any value and $w_i\geq 0$, 2) if $u_i<0$ then $v_i\leq 0$ and $w_i$ can take any value, and 3) if $u_i=0$ then $v_i\leq 0$ and $w_i\geq 0$. Resuming these three cases to symbols we have: $(\dots, +,\dots)=(\dots, *,\dots)+_{sc}(\dots, \oplus,\dots)$, $(\dots, -,\dots)=(\dots, \ominus,\dots)+_{sc}(\dots, *,\dots)$, $(\dots, 0,\dots)=(\dots, \ominus,\dots)+_{sc}(\dots, \oplus,\dots)$, where the symbol $\ominus $ means that the corresponding integer is non positive, the symbol $\oplus $ non negative and the symbol $*$ means that it can take any value.

\section{Markov complexity and Bouquet ideals}

\medskip

To any toric ideal $I_A\subseteq \Bbbk[x_1,\ldots, x_n]$, encoded by an integer matrix $A\in\ZZ^{m\times n}$,
one can associate another toric ideal, called the ``bouquet ideal",
denoted by $I_{A_B}$. For details of this construction see \cite[Section 1]{PTV},
and how relate in general the properties of the bouquet toric ideal $I_{A_B}$ to
the starting toric ideal $I_A$, see \cite{PTV, PTV2}. Next, we revisit a few technical
details of the bouquet construction, needed for proofs. We recall that two column
vectors $\ab_i,\ab_j$ are in the same bouquet if there exists a vector $\vb\in\ZZ^m$
such that the matrix product $\vb\cdot A$ is a vector with only two nonzero components,
those appearing on the $i$-th and $j$-th positions. In other words, this is equivalent
to saying that the dot products $\vb\cdot \ab_k=0$ for all $k\neq i,j$ and $\vb\cdot\ab_i\neq 0$, $\vb\cdot\ab_j\neq 0$. Then, if $A$ has $q$ bouquets then, without loss of generality, we may assume that there exist positive integers $i_1, i_2, \ldots, i_q$ such that the first $i_1$ column vectors $\ab_1,\ldots,\ab_{i_1}$ belong to the first bouquet, the next $i_2$ column vectors belong to the second bouquet, and so on the last $i_q$ vectors belong to the $q$-th bouquet, thus in particular $i_1+\cdots+i_q=n$.

Next, we recall from \cite{PTV}, the definition of the \emph{bouquet-index-encoding vector} $\cb_B$,
since the proof of Theorem~\ref{latter} requires the basics of its definition and properties.
If the bouquet $B$ is free then we set $\cb_B\in\ZZ^n$ to be any nonzero vector such that
$\supp(\cb_B)=\{i: \ab_i\in B\}$ and with the property that the first nonzero coordinate is positive.
For a  non-free bouquet $B$ of $A$, consider the Gale transforms of the elements in $B$.
All the elements are nonzero and pairwise linearly dependent, therefore there exists a nonzero
coordinate $j$ in all of them. Let $g_j=\gcd(G({\bf a}_i)_j| \ {\bf a}_i\in B)$ and fix the smallest
integer $i_0$ such that ${\bf a}_{i_0}\in B$. Let ${\bf c}_B$  be the vector in $\ZZ^n$ whose $i$-th coordinate
is $0$ if ${\bf a}_i\notin B$, and is $\varepsilon_{i_0j}G({\bf a}_i)_j/g_j$ if ${\bf a}_i \in B$, where $\varepsilon_{i_0j}$ represents the sign of the integer $G({\bf a}_{i_0})_j$.
 Thus the $\supp({\bf c}_B)=\{i: {\bf a}_i\in B\}$. Note that the choice of $i_0$ implies that the first nonzero coordinate of ${\bf c}_B$ is positive. Since each ${\bf a}_i$ belongs to  exactly one bouquet the supports of the vectors ${\bf c}_{B_i}$ are pairwise disjoint. In addition, $\cup_i\supp(\cb_{B_i})=[n]$.
For each bouquet $B$ we define the vector
 $\ab_B=\sum_{i=1}^n (c_B)_i\ab_i\in\ZZ^m$,
where $(c_B)_i$ is the $i$-th component of the vector $\cb_B$. The matrix $A_B$ whose column vectors are the vectors $\ab_B$
corresponding to the bouquets of $A$ is called the {\em bouquet matrix} of $A$.

Finally, we recall that (see \cite[Theorem 1.9]{PTV}) the matrix $A_B\in\ZZ^{m\times q}$ has its kernel $\Ker_{\ZZ}(A_B)$ isomorphic to $\Ker_{\ZZ}(A)$ via the isomorphism $D:\Ker_{\ZZ}(A_B)\mapsto\Ker_{\ZZ}(A)$ given by
\[
D(u_1,\ldots,u_q)=(c_{11}u_1,\ldots,c_{1i_1}u_1,c_{21}u_2,\ldots,c_{2i_2}u_2,\ldots,c_{q1}u_q,\ldots,c_{qi_q}u_q),
\]
where the integers $c_{11},c_{21},\ldots,c_{q1}$ are the first nonzero coordinates of the vectors $\cb_{B_1}$, $\cb_{B_2}, \ldots, \cb_{B_q}$, and thus, all positive. Note that the vector in the image has $n$ coordinates, and that $D$ can be naturally extended to a linear map $D':\QQ^q\mapsto\QQ^n$. The inverse of the map $D$ we are going to denote it by $T$, and $T:\Ker_{\ZZ}(A)\mapsto\Ker_{\ZZ}(A_B)$ is defined
\[
T(v_{11},\ldots,v_{1i_1},v_{21},\ldots,v_{2i_2},\ldots,v_{q1},\ldots,v_{qi_q})=(v_{11}/c_{11},\ldots,v_{i1}/c_{i1},\ldots,v_{q1}/c_{q1}),
\]
and admits, as $D$, a natural extension to a linear map $T':\QQ^n\mapsto\QQ^q$. Here, we used in the definition of $D$ and $T$, the assumption that the first $i_1$ column vectors of $A$ belong to the first bouquet, and so forth.

The next theorem generalizes \cite[Theorem 1.11]{PTV} for Markov bases.

\begin{Theorem}
 \label{Markov}
If $\{{\bf u}_1, \ldots, {\bf u}_t\}$ is a minimal Markov basis for $A$ then $\{T({\bf u}_1), \ldots, T({\bf u}_t)\}$ is a Markov basis for $A_B$, not necessarily minimal.
\end{Theorem}
\begin{proof}
Let $\MM =\{{\bf u}_1, \ldots, {\bf u}_t\}$ be a minimal Markov basis of $A$. Consider ${\bf w}, {\bf u} \in \mathbb{N}^q$ such that ${\bf w}-{\bf u} \in \Ker_{\ZZ}(A_B)$, and we look then at the elements $D'({\bf w})^++D'({\bf u})^-$ and $D'({\bf w})^-+D'({\bf u})^+$ of $\mathbb{N}^n$. Since $$(D'({\bf w})^++D'({\bf u})^-)-(D'({\bf w})^-+D'({\bf u})^+)=D'({\bf w}-{\bf u})=D(\wb-\ub)\in  \Ker_{\ZZ}(A),$$
then there exists a subset $\{{\bf v}_i : i = 1, \ldots,p\}$ of $\MM$ that connects $D'({\bf w})^++D'({\bf u})^-$ to $D'({\bf w})^-+D'({\bf u})^+$. By definition, this means that $(D'({\bf w})^++D'({\bf u})^-- {{\sum^l_{i=1}}}{{\bf v}_i})\in \NN^n$  for all $1\leq l\leq p$, and
$$(D'({\bf w})^++D'({\bf u})^-)- (D'({\bf w})^-+D'({\bf u})^+)(=D({\bf w}-{\bf u}))=\sum^p_{i=1}{\bf v}_i.$$ Applying $T'$ to the previous two equalities we have $T'(D'({\bf w})^++D'({\bf u})^-)- {{\sum^l_{i=1}}}{T({\bf v}_i})\in \NN^q$ and $T'(D({\bf w}-{\bf u}))=T(D({\bf w}-{\bf u}))=\sum^p_{i=1}{T(\bf v}_i )$ (here we used that $T'(\vb_i)=T(\vb_i)$ since $\vb_i\in\Ker_{\ZZ}(A)$).

 Note that $T'(D'({\bf w})^+)=\wb$ since  ${\bf w} \in \mathbb{N}^q$ and $c_{11},\ldots,c_{q1}$ are positive. In addition, $T'(D'({\bf u})^-)={\bf 0}$ since the $i1$-th coordinate of $D'(\ub)^-$ is $0$ because $D'(\ub)$ has all $i1$-th coordinates nonnegative (${\bf u} \in \mathbb{N}^q$ and $c_{i1}>0$ for all $i$). Hence $T'(D'({\bf w})^++D'({\bf u})^-)=\wb+{\bf 0}=\wb$, and therefore we obtain ${\bf w}-{{\sum^l_{i=1}}}{T({\bf v}_i})\in \NN^q$ and ${\bf w}-{\bf u}=\sum^p_{i=1}{T(\bf v}_i )$. The latter relations imply that $\{T({\bf u}_1), \ldots, T({\bf u}_t)\}$ is a Markov basis for $A_B$.\qed

\end{proof}

The following example shows that the minimality of a Markov basis is not preserved in general after applying $T$.

\begin{Example}\label{no_minimal}
{\em  Consider the integer matrix
 \[
 A={\footnotesize\begin{pmatrix}
3 & 3 & 4 & 5 \\
2 & 3 & 0 & 0
\end{pmatrix}}.
\]
Then $A$ has three bouquets, with the first two columns being in the same bouquet while the other two are bouquets by themselves, and $\cb_{B_1}=(3,-2,0,0)$, $\cb_{B_2}=(0,0,1,0)$, $\cb_{B_3}=(0,0,0,1)$ (this implies that $c_{11}=3, c_{21}=1, c_{31}=1$). Hence the bouquet matrix is (see \cite[Example 1.10]{PTV} for detailed computation of the matrix $A_B$)
 \[
A_B={\footnotesize \begin{pmatrix}
3 & 4 & 5 \\
0 & 0 & 0
\end{pmatrix}}.
\]
The map $T$ is given by
$T(u_{11}, u_{12}, u_{21}, u_{31})=(u_{11}/3, u_{21}/1, u_{31}/1)$. Note that the set \{$(0,0,5,-4)$, $(3,-2,-2,1)$, $(3,-2,3,-3)$, $(6,-4,1,-2)$, $(9,-6,-1,-1)$\} represents a minimal Markov basis for $A$, and therefore by Theorem~\ref{Markov} the vectors  $T(0,0,5,-4)=(0,5,-4)$, $T(3,-2,-2,1)=(1,-2,1)$, $T(3,-2,3,-3)=(1,3,-3)$, $T(6,-4,1,-2)=(2,1,-2)$, $T(9,-6,-1,-1)=(3,-1,-1)$ form a Markov basis for $A_B$, but not minimal. Indeed, one can easily check that a minimal Markov basis of $A_B$ consists only of the vectors $(1,-2,1)$, $(2,1,-2)$, $(3,-1,-1)$.}
\end{Example}

\begin{Corollary}
 \label{ind}
 Suppose that $\Ker_{\ZZ}(A_B)\cap\NN^n=\{{\bf 0}\}$ then if ${\bf u}$ is an indispensable element of $A_B$ then $D({\bf u})$ is an indispensable element of $A$.
\end{Corollary}
\begin{proof}
Let $\{{\bf u}_1, \ldots, {\bf u}_t\}$ be a minimal Markov basis for $A$. 
Applying Theorem~\ref{Markov} we obtain that \{$T({\bf u}_1),\ldots,T({\bf u}_t)\}$ is a Markov basis for $A_B$. 
By the homogeneous Nakayama Lemma a subset of it is a minimal Markov basis for $A_B$. 
Since ${\bf u}$ is an indispensable element of $A_B$ then there exists an $i$ such that $T({\bf u}_i)={\bf u}$. 
Then applying $D$ we have $D({\bf u})=D(T({\bf u}_i))=\ub_i$, which means that  $D({\bf u})$ 
is an indispensable element of $A$.\qed

\end{proof}

If $A_B$ is the matrix obtained by the bouquet construction from $A$, then one can easily see that when passing to the $r$-th Lawrence liftings of these matrices, the map $D_{(r)}:\Ker_{\ZZ}(A_B^{(r)})\mapsto \Ker_{\ZZ}(A^{(r)})$ defined by
\[D_{(r)} {\footnotesize\begin{pmatrix}
{\bf u}_1\\
{\bf u}_2 \\
\vdots   \\
{\bf u}_r
\end{pmatrix}} = {\footnotesize\begin{pmatrix}
D({\bf u}_1)\\
D({\bf u}_2) \\
\vdots   \\
D({\bf u}_r)
\end{pmatrix},}
\]
is well defined since $D$ is linear (${\bf u}_1+\cdots +{\bf u}_r={\bf 0}$ implies that $D({\bf u}_1)+\cdots +D({\bf u}_r)={\bf 0}$). Similarly one can define for any $r$ the map $T_{(r)}:\Ker_{\ZZ}(A^{(r)})\mapsto \Ker_{\ZZ}(A_B^{(r)})$, starting from $T$. On the other hand, we can also discuss about the bouquet matrix of the $r$-th Lawrence lifting of $A$, that is $(A^{(r)})_B$, and wonder how it relates to the one of $A_B^{(r)}$.

\begin{Theorem}
 \label{latter}
Let $A_B$ be the bouquet matrix of $A$ and $r\geq 3$ an integer. Then, $$\Ker_{\ZZ}({(A^{(r)})_B})=\Ker_{\ZZ}({A_B^{(r)}}),$$ and in particular $I_{(A^{(r)})_B}=I_{({A_B^{(r)}})}$.
\end{Theorem}

\begin{proof}
Let $\ab_i(s)$ denote the $i+(s-1)n$ column of $A^{(r)}$ for $s=1,\ldots,r$. We prove that each bouquet $B$ of $A$ gives rise to exactly $r$ subbouquets of $A^{(r)}$, and those $r$ subbouquets of $A$ are bouquets if $B$ is non-free.

Indeed, we first see that if $\ab_j,\ab_k$ belong to the same bouquet, then there exists a vector $\vb\in\ZZ^m$
such that $\vb\cdot A$ has the only nonzero components, $\vb\cdot\ab_j=l_j$ and $\vb\cdot\ab_k=l_k$, on positions $j$ and $k$, respectively. Then the vector $\wb\in\ZZ^{rm+n}$, whose $i+(s-1)m$-th component is $v_i$ (the $i$-th component of $\vb$), for all $i=1,\ldots,m$, and all other components are zero
$\wb=({\bf 0},
{\bf 0},
\ldots,
{\vb},
\ldots
{\bf 0}),$
has the property $(*)$ that ${\wb}\cdot\ab_j(s)=\vb\cdot\ab_j=l_j\neq 0$, ${\wb}\cdot\ab_k(s)=\vb\cdot\ab_k=l_k\neq 0$,
${\wb}\cdot\ab_i(s)=\vb\cdot\ab_i=0$ for all $i\neq j,k$, and $\wb\cdot\ab_i(p)=0$ for all $p\neq s$ and all $i=1,\ldots,n$. Thus $\ab_j(s)$ and $\ab_k(s)$ belong to the same bouquet.
Therefore, if $B=\{\ab_{j_i+1},\ldots ,\ab_{j_{i+1}}\}$ is a bouquet of $A$ then $\ab_{j_i+1}(s), \ldots ,\ab_{j_{i+1}}(s)$ belong to the same bouquet of $A^{(r)}$, for any $1\leq s\leq r$.

 It only remains to prove that if $B$ is a non-free bouquet of $A$ then $\{\ab_{j_i+1}(s), \ldots ,\ab_{j_{i+1}}(s)\}$ is a
bouquet of $A^{(r)}$ for any $1\leq s\leq r$. To show this we have to analyze two cases: 1) $a_j\in B$ and $a_k\notin B$ then $a_j(s)$ and $a_k(s)$ are not in the same bouquet for any $s$, and 2) $a_j(s)$ and $a_k(p)$ are not in the same bouquet for any $k$ and any $p\neq s$.

In the first case, since $\ab_j$ belongs to a non-free bouquet and $\ab_j,\ab_k$ do not belong to the same bouquet, then there exists $\ub\in\Ker_{\ZZ}(A)$ such that its $j$-th component $u_j\neq 0$ and its $k$-th component $u_k=0$. Arguing by contradiction, assume that there exists a vector $\zb\in\ZZ^{rm+n}$ such that ${\zb}\cdot A^{(r)}$ has the only nonzero components on positions $j+(s-1)n$ and $k+(s-1)n$. Let $\yb\in\ZZ^{rn}$ be the vector in $\Ker_{\ZZ}(A^{(r)})$, whose $s$-th row vector is $\ub$ constructed above, and all other row vectors are {\bf 0} except one, which is $-\ub$, and placed arbitrarily (it is possible since $r\geq 3$). Then we obtain that ${\zb}\cdot (A^{(r)}\cdot\yb)={\zb}\cdot {\bf 0}=0$, while $({\zb}\cdot A^{(r)})\cdot \yb\neq 0$ since the matrix product equals the product of the $j+(s-1)n$-th component of ${\zb}\cdot A^{(r)}$ and $u_j$, which is nonzero by construction. Therefore we obtain a contradiction, and the first case is proved.

In the second case, since $\ab_j$ belongs to a non-free bouquet then there exists $\ub\in\Ker_{\ZZ}(A)$ such that its $j$-th component $u_j\neq 0$. Arguing by contradiction and assuming that $a_j(s)$ and $a_k(p)$ are in the same bouquet for some $k$ and some $p\neq s$, then there exists a vector $\zb\in\ZZ^{rm+n}$ such that ${\zb}\cdot A^{(r)}$ has the only nonzero components on positions $j+(s-1)n$ and $k+(p-1)n$. Now, we construct $\yb\in\ZZ^{rn}$ to be the vector in $\Ker_{\ZZ}(A^{(r)})$, whose $s$-th row  vector is $\ub$ constructed above, the other non-zero row vector $-\ub$ is not the $p$-th one, and all the remaining row vectors (including the $p$-th) are {\bf 0} (it is possible since $r\geq 3$).
Then, as in the first case we obtain that ${\zb}\cdot (A^{(r)}\cdot\yb)=0$, while $({\zb}\cdot A^{(r)})\cdot \yb\neq 0$, a contradiction.  Therefore, the second case is also proved, and consequently we have proved that if $B$ is a non-free bouquet of $A$ then $B(s)$ is a (non-free) bouquet of $A^{(r)}$ for all $s=1,\ldots,r$.

We denote by $B_1,\ldots,B_q$ the bouquets of $A$,
and by $B_1(1), \ldots, B_q(1)$, $B_1(2), \ldots, B_q(2)$, $\ldots$,  $B_1(r), \ldots, B_q(r)$ the subbouquets (as shown before) of $A^{(r)}$.
Then, as explained at the beginning of this section and using the property $(*)$ written above 
we have $l_j G(\ab_j(s))+l_k G(\ab_k(s))={\bf 0}$ where $\ab_j,\ab_k$ belong 
to the same bouquet and $\vb$ is the vector defined before $(*)$ 
such that $\vb\cdot\ab_j=l_j$ and $\vb\cdot\ab_k=l_k$. 
It follows from the definition of $\cb_B$ that  $l_j \cb_{B_j}+l_k \cb_{B_k}={\bf 0}$ and 
therefore also $l_j\cb_{B_j(s)}+l_k \cb_{B_k(s)}={\bf 0}$. Thus for each $i=1,\ldots, q$ 
the corresponding $\cb_B$-vectors are
$\cb_{B_i(1)}=\footnotesize {(\cb_{B_i}, {\bf 0}, \ldots, {\bf 0})}, \ \ \ldots, \ \ \cb_{B_i(r)}=\footnotesize{({\bf 0}, {\bf 0}, \ldots, \cb_{B_i})}.$
Then, it follows from the definition of the $\ab_B$-vectors that $\ab_{B_i(s)}$ is the transpose of the matrix
$( {\bf 0}, {\bf 0},
\cdots   ,
{\bf a_{B_i}},
\cdots ,
{\bf 0},
\cb_{B_i}^T
),$
for any $i=1,\ldots,q$, where the vector $\ab_{B_i}$ is positioned as the $s$-th row vector for any $s=1,\ldots,r$. This implies that the subbouquet matrix of $A^{(r)}$ is
\[
(A^{(r)})_B={\footnotesize \begin{array}{c} \overbrace{\quad\quad\quad \quad \quad\ \ }^{r-\textrm{times}} \\
 \begin{pmatrix}
\ A_B\ & 0 &  \hdots & 0 \\
0 &\ A_B\ & \hdots & 0 \\
\vdots & \vdots & \ddots & \vdots \\
0 & 0 & \hdots &\ A_B\ \\
C & C & \cdots & C
\end{pmatrix}\end{array},}
\]
where $C$ is a matrix $n\times q$ whose $i$-th column is the transpose of $\cb_{B_i}$, that is $C=[\cb_{B_1}^T, \cb_{B_2}^T,\ldots,\cb_{B_q}^T]\in\ZZ^{n\times q}$. Note now that the subbouquet matrix of the $r$-th Lawrence lifting of $A$ and the $r$-th Lawrence lifting of the bouquet matrix $A_B$ of $A$, that is

\[
 A_B^{(r)}={\footnotesize \begin{array}{c} \overbrace{\quad\quad\quad \quad \quad\ \ }^{r-\textrm{times}} \\
 \begin{pmatrix}
\ A_B\ & 0 &  \hdots & 0 \\
0 &\ A_B\ & \hdots & 0 \\
\vdots & \vdots & \ddots & \vdots \\
0 & 0 & \hdots &\ A_B\ \\
I_q & I_q & \cdots & I_q
\end{pmatrix}\end{array},}
\]
have the same kernels, which completes the proof. Indeed, to prove the latter, let $\wb={\footnotesize\begin{pmatrix}
\wb_1\\
\wb_2 \\
\vdots   \\
\wb_r
\end{pmatrix}}\in\Ker_{\ZZ}\left (({A^{(r)}})_B \right)$, where each $\wb_i\in\ZZ^q$. This is equivalent to $\wb_i\in\Ker_{\ZZ}(A_B)$ for each $i=1,\ldots,r$ and $C\cdot\wb_1+C\cdot\wb_2+\ldots+C\cdot\wb_r=\bf 0$. Note now that since the columns of $C$ are the transpose of the vectors $\cb_{B_1},\ldots,\cb_{B_q}$ then the matrix product $C\cdot \wb_j$ is just $D(\wb_j)$, where $D:\Ker_Z(A_B)\to\Ker_Z(A)$ is the isomorphism defined at the beginning of the section. Therefore $C\cdot\wb_1+C\cdot\wb_2+\ldots+C\cdot\wb_r=\bf 0$ is equivalent to $\wb_1+\ldots+\wb_q=\bf 0$ since $D$ is an isomorphism, and consequently $\wb\in\Ker_{\ZZ}\left (({A^{(r)}})_B \right)$ is equivalent to $\wb\in\Ker_{\ZZ}(A_B^{(r)})$, and the proof is complete.   \qed

\end{proof}
\begin{Remark}
{\em Note that in the above proof the free bouquet of $A$, if it exists, gives rise to $r$ proper subbouquets of $A^{(r)}$ and whose union is a bouquet of $A^{(r)}$.}
\end{Remark}

\begin{Theorem}
 \label{Markovcomplexity}
The Markov complexity of $A$ is greater than or equal to the Markov complexity of its bouquet matrix $A_B$.
\end{Theorem}
\begin{proof}
 First note that both maps $T_{(r)}$ and $D_{(r)}$ preserve the type, since $\ub\neq 0$ is equivalent to $D(\ub)\neq 0$. Suppose that $A$ has Markov complexity $m$. Then there exists an element of type $m$ in a minimal Markov basis of $A^{(r)}$ for some $r$ and all elements in the minimal Markov bases of $A^{(s)}$ for any $s$ have type less than or equal to $m$. Then according to Theorem~\ref{Markov} the images of the Markov basis vectors under  $T_{(s)}$ form a Markov basis for $A_B^{(s)}$, not necessarily minimal. So there exists a minimal Markov basis for $A_B^{(s)}$ with types smaller than or equal to $m$ for every $s$. But according to \cite[Theorem 3.3]{HaraThomaVladoiu15} the types are preserved in all minimal Markov bases. Therefore the Markov complexity of $A_B$ is less than or equal to $m$. \qed

\end{proof}

\section{The family of matrices $ A_s $}
\label{}

Let $s\geq 3$ and consider the family of matrices
\[
A_s={\footnotesize \begin{pmatrix}
0 & 1 & s-1  & s \\
1 & 1 & 1 & 1
\end{pmatrix}}.
\]

First we recall that for an arbitrary integer matrix $A$ and a vector $\wb\in\Ker_{\ZZ}(A)$, the fiber corresponding to $\wb$, denoted by $\mathcal{F}_{\wb}$, is the set $\{\tb\in\NN^n: \wb^{+}-\tb\in\Ker_{\ZZ}(A)\}$, which is a finite set iff $I_A$ is positively graded. It was proved in \cite[Corollary 4.10]{CKT} that the vector  $\wb\in\Ker_{\ZZ}(A)$ is indispensable if and only if $\mathcal{F}_{\wb}=\{\wb^{+},\wb^{-}\}$. With this observation we can identify in the next lemma some indispensable elements of $\Ker_{\ZZ}(A_s)$.

\begin{Lemma}\label{1}
The elements $(1,-1,-1,1), (2-s, s-1, -1, 0)$ and $(0,-1, s-1, 2-s)$ of $\Ker_{\ZZ}(A_s)$ are indispensable.
\end{Lemma}
\begin{proof} Let $\ub=(1,-1,-1,1)$ and $\tb=(\alpha, \beta, \gamma, \delta)$ an arbitrary vector of $\mathcal{F}_{\ub}$. Since $\tb\in\mathcal{F}_{\ub}$ then $A_s\ub^+=A_s\tb$, which implies the vectorial equality
\[
{\footnotesize \begin{bmatrix}
s \\
2 \end{bmatrix}} = \alpha{\footnotesize \begin{bmatrix}
0 \\
1 \end{bmatrix}} + \beta {\footnotesize \begin{bmatrix}
1 \\
1 \end{bmatrix}}+ \gamma {\footnotesize \begin{bmatrix}
s-1 \\
1 \end{bmatrix}} + \delta {\footnotesize \begin{bmatrix}
s \\
1 \end{bmatrix}}.
\]

Note that $\delta \leq 1$. If $\delta=1$ then we necessarily have $\alpha =1$ and $\beta =\gamma =0$ implying that $\tb=\ub^+$. Otherwise $\delta =0$, and then we obtain $\alpha =0$ and $\beta =\gamma =1$, hence $\tb=\ub^-$. We conclude that $\mathcal{F}_{\ub}=\{\ub^+,\ub^-\}$, and by the previous remark we get that $(1, -1, -1, 1)$ is indispensable.

Similarly, if $\ub= (2-s, s-1, -1, 0)$ then we compute $\mathcal{F}_{\ub}$. Let $\tb=(\alpha, \beta, \gamma, \delta)\in \mathcal{F}_{\ub}$, and by the definition of the fiber, as before, we get the following vectorial equality
\[
{\footnotesize \begin{bmatrix}
s-1 \\
s-1 \end{bmatrix}} = \alpha{\footnotesize \begin{bmatrix}
0 \\
1 \end{bmatrix}} + \beta {\footnotesize \begin{bmatrix}
1 \\
1 \end{bmatrix}}+ \gamma {\footnotesize \begin{bmatrix}
s-1 \\
1 \end{bmatrix}} + \delta {\footnotesize \begin{bmatrix}
s \\
1 \end{bmatrix}}.
\]

Note that $\delta =0$ and $\gamma \leq 1$. If $\gamma =1$ then necessary $\alpha =s-2 $ and $\beta =0$, which implies that $\tb=\ub^-$. Otherwise $\gamma =0$, and then necessarily $\alpha =0$ and $\beta =s-1$, which means that $\tb=\ub^+$. Therefore $\mathcal{F}_{\ub}=\{\ub^+,\ub^-\}$, and again by previous remark we conclude that $(2-s, s-1, -1, 0)$ is indispensable. Finally, the element $(0,-1, s-1, 2-s)$ is indispensable with an identical proof as in the previous case. \qed

\end{proof}

\section{Markov complexity of $2\times 4$ matrices}
\label{Section: MainResult}

In this section, we prove one of the main results of this paper regarding the unboundedness of the Markov complexity
of  $2\times 4$ matrices.  To prove this, we need to find a family of $2\times 4$ matrices that depend on a parameter $r$, such that the Markov complexity of $A_r$, the $r^{th}$ member of the family, is at least $r$. The latter reduces to finding an element of type $r$ that belongs to $\MM(A_r^{(r)})$.

\begin{Theorem} \label{24}
\label{large} The $2\times 4$ matrices may have arbitrarily large Markov complexity.
\end{Theorem}

\begin{proof}
We will show that the type $s$ element
\[
{\bf u} ={\footnotesize \begin{pmatrix}
1 & -1 & -1  & 1 \\
1 & -1 & -1  & 1 \\
\vdots & \vdots & \vdots & \vdots \\
1 & -1 & -1  & 1 \\
0 & -1 & s-1 & 2-s \\
2-s & s- 1 & -1 & 0
\end{pmatrix}}
\]
 belongs to every Markov basis of $A_s^{(s)}$, where $A_s={\footnotesize \begin{pmatrix}
0 & 1 & s-1  & s \\
1 & 1 & 1 & 1
\end{pmatrix}}$, and consequently $m(A_s)\geq s$. In other words, we want to prove that $\ub$ is an indispensable element of $A_s^{(s)}$. Note first that $\ub\in\Ker_{\ZZ}(A_s^{(s)})$ since every row  of $\ub$ is in $ \Ker_{\ZZ}(A_s)$ and the sum of each column is zero.

Let us assume on the contrary that the element ${\bf u}$ is not indispensable. Using Proposition~\ref{indispensable} this implies that ${\bf u}$ admits a proper semiconformal
decomposition ${\bf u}={\bf v}+_{sc} {\bf w}$, where ${\bf u},{\bf v},{\bf w} \in  \Ker_{\ZZ}(A_s^{(s)})$ such that
$$
{\bf v}_{ij}>0 \Rightarrow {\bf w}_{ij} \geq 0 \text{ and }
{\bf w}_{ij}<0 \Rightarrow {\bf v}_{ij} \leq 0 \text{, }
$$
for any $1 \leq i \leq s, 1 \leq j \leq 4$. In terms of signs, for each row of the vector ${\bf u}$ we have the following semiconformal decomposition
\begin{eqnarray*}
(1,-1,-1,1) & = & ( * , \ominus , \ominus , * ) +_{sc} ( \oplus ,* , * , \oplus) \\
(0, -1, s-1, 2-s) & = & (\ominus,\ominus,*,\ominus) +_{sc} (\oplus,*,\oplus,*)\\
(2-s, s-1, -1, 0) & = & (\ominus,*,\ominus,\ominus) +_{sc} (*,\oplus,*,\oplus) \text{.}
\end{eqnarray*}

  Since ${\bf u}={\bf v}+_{sc} {\bf w}$ is a semiconformal decomposition of ${\bf u}$, then the sign pattern of the elements ${\bf v}, {\bf w}$ is:
 \[
{\bf u} ={\footnotesize \begin{pmatrix}
1 & -1 & -1  & 1 \\
1 & -1 & -1  & 1 \\
\vdots & \vdots & \vdots & \vdots \\
1 & -1 & -1  & 1 \\
0 & -1 & s-1 & 2-s \\
2-s & s- 1 & -1 & 0
\end{pmatrix}
= \begin{pmatrix}
* & \ominus & \ominus  & * \\
* & \ominus & \ominus & * \\
\vdots & \vdots & \vdots & \vdots \\
* & \ominus & \ominus  & * \\
\ominus & \ominus & * & \ominus \\
 \ominus &  * & \ominus & \ominus
\end{pmatrix}
+_{sc}
\begin{pmatrix}
\oplus & * & *  & \oplus \\
\oplus & * & *  & \oplus \\
\vdots & \vdots & \vdots & \vdots \\
\oplus & * & *  & \oplus \\
\oplus & * & \oplus & * \\
 * &  \oplus & * & \oplus
\end{pmatrix}}.
\]

According to Lemma~\ref{1} the element $(1,-1,-1,1)$ is indispensable,
which means that it has only two semiconformal decompositions, see Proposition \ref{indispensable}, corresponding to the following two cases.

\medskip

{\bf First case:} the first row of ${\bf v}$ is $ (1, -1, -1, 1)$ and  the first row of ${\bf w}$ is $(0, 0, 0, 0)$, i.e.

\[
{\bf u} ={\footnotesize \begin{pmatrix}
1 & -1 & -1  & 1 \\
1 & -1 & -1  & 1 \\
\vdots & \vdots & \vdots & \vdots \\
1 & -1 & -1  & 1 \\
0 & -1 & s-1 & 2-s \\
2-s & s- 1 & -1 & 0
\end{pmatrix}
= \begin{pmatrix}
1 & -1 & -1  & 1 \\
* & \ominus & \ominus & * \\
\vdots & \vdots & \vdots & \vdots \\
* & \ominus & \ominus  & * \\
\ominus & \ominus & * & \ominus \\
 \ominus & * & \ominus & \ominus
\end{pmatrix}
+_{sc}
\begin{pmatrix}
0 & 0 & 0  & 0 \\
\oplus & * & *  & \oplus \\
\vdots & \vdots & \vdots & \vdots \\
\oplus & * & *  & \oplus \\
\oplus & * & \oplus & * \\
 * &  \oplus & * & \oplus
\end{pmatrix} }.
\]

Considering that the sum of every column of $\vb$ should be zero, we conclude that $v_{s,2}$, the last element of the second column of ${\bf v}$, and $v_{s-1,3}$, the
$s-1$ element of the third column of ${\bf v}$ are both positive, so

\[
{\bf u} ={\footnotesize \begin{pmatrix}
1 & -1 & -1  & 1 \\
1 & -1 & -1  & 1 \\
\vdots & \vdots & \vdots & \vdots \\
1 & -1 & -1  & 1 \\
0 & -1 & s-1 & 2-s \\
2-s & s- 1 & -1 & 0
\end{pmatrix}
= \begin{pmatrix}
1 & -1 & -1  & 1 \\
* & \ominus & \ominus & * \\
\vdots & \vdots & \vdots & \vdots \\
* & \ominus & \ominus  & * \\
\ominus & \ominus & + & \ominus \\
 \ominus & + & \ominus & \ominus
\end{pmatrix}
+_{sc}
\begin{pmatrix}
0 & 0 & 0  & 0 \\
\oplus & * & *  & \oplus \\
\vdots & \vdots & \vdots & \vdots \\
\oplus & * & *  & \oplus \\
\oplus & * & \oplus & * \\
 * &  \oplus & * & \oplus
\end{pmatrix}}.
\]

It follows from Lemma~\ref{1} that the elements $(0,-1, s-1, 2-s)$ and $(2-s, s-1, -1, 0)$ of $\Ker_{\ZZ}(A_s)$ are indispensable, therefore they have each only two semiconformal decompositions. This implies that the last two rows of $\wb$ are zero, since we have just noticed that the last two rows of $\vb$ have each a positive component. Hence

\[
{\bf u} ={\footnotesize \begin{pmatrix}
1 & -1 & -1  & 1 \\
1 & -1 & -1  & 1 \\
\vdots & \vdots & \vdots & \vdots \\
1 & -1 & -1  & 1 \\
0 & -1 & s-1 & 2-s \\
2-s & s-1 & -1 & 0
\end{pmatrix}
= \begin{pmatrix}
1 & -1 & -1  & 1 \\
* & \ominus & \ominus & * \\
\vdots & \vdots & \vdots & \vdots \\
* & \ominus & \ominus  & * \\
0 & -1 & s-1 & 2-s \\
2-s & s- 1 & -1 & 0
\end{pmatrix}
+_{sc}
\begin{pmatrix}
0 & 0 & 0  & 0 \\
\oplus & * & *  & \oplus \\
\vdots & \vdots & \vdots & \vdots \\
\oplus & * & *  & \oplus \\
0 & 0 & 0 & 0 \\
 0 &  0 & 0 & 0
\end{pmatrix} }.
\]

The remaining $s-3$ rows of ${\bf w}$ are either $(1, -1, -1, 1)$ or $(0, 0, 0, 0)$ and taking into account that
the sum of every column should be zero, then we get that all rows of $\wb$
should be  $(0, 0, 0, 0)$ thus  $\wb$ is the zero matrix, a contradiction to our assumption that $\ub=\vb+_{sc}\wb$ is a proper semiconformal decomposition.
\medskip

{\bf Second case:} the first row of ${\bf v}$ is $(0, 0, 0, 0)$ and  the first row of ${\bf w}$ is $(1, -1, -1, 1)$, i.e.

\[
{\bf u} ={\footnotesize \begin{pmatrix}
1 & -1 & -1  & 1 \\
1 & -1 & -1  & 1 \\
\vdots & \vdots & \vdots & \vdots \\
1 & -1 & -1  & 1 \\
0 & -1 & s-1 & 2-s \\
2-s & s- 1 & -1 & 0
\end{pmatrix}
= \begin{pmatrix}
0 & 0 & 0  & 0 \\
* & \ominus & \ominus & * \\
\vdots & \vdots & \vdots & \vdots \\
* & \ominus & \ominus  & * \\
\ominus & \ominus &  * & \ominus \\
 \ominus &  * & \ominus & \ominus
\end{pmatrix}
+_{sc}
\begin{pmatrix}
1 & -1 & -1  & 1 \\
\oplus & * & *  & \oplus \\
\vdots & \vdots & \vdots & \vdots \\
\oplus & * & *  & \oplus \\
\oplus & * & \oplus &  * \\
 * &  \oplus & * & \oplus
\end{pmatrix}}.
\]

Considering that the sum of every column of $\wb$ should be zero, then we conclude that $w_{s,1}$, the last element of the first column of ${\bf w}$, and $w_{s-1,4}$, the $s-1$ element
of the fourth column of ${\bf w}$, are both negative, so

\[
{\bf u} ={\footnotesize \begin{pmatrix}
1 & -1 & -1  & 1 \\
1 & -1 & -1  & 1 \\
\vdots & \vdots & \vdots & \vdots \\
1 & -1 & -1  & 1 \\
0 & -1 & s-1 & 2-s \\
2-s & s- 1 & -1 & 0
\end{pmatrix}
= \begin{pmatrix}
0 & 0 & 0  & 0 \\
* & \ominus & \ominus & * \\
\vdots & \vdots & \vdots & \vdots \\
* & \ominus & \ominus  & * \\
\ominus & \ominus & * & \ominus \\
 \ominus & * & \ominus & \ominus
\end{pmatrix}
+_{sc}
\begin{pmatrix}
1 & -1 & -1  & 1 \\
\oplus & * & *  & \oplus \\
\vdots & \vdots & \vdots & \vdots \\
\oplus & * & *  & \oplus \\
\oplus & * & \oplus & - \\
 - &  \oplus & * & \oplus
\end{pmatrix}}.
\]

Since $(0,-1, s-1, 2-s)$ and $(2-s, s-1, -1, 0)$ are indispensable elements of $\Ker_{\ZZ}(A_s)$, then they have each only two semiconformal decompositions. We obtain that the last two rows of $\vb$ are zero, since we have just noticed that the last two rows of $\wb$ have each a negative component. Hence

\[
{\bf u} ={\footnotesize \begin{pmatrix}
1 & -1 & -1  & 1 \\
1 & -1 & -1  & 1 \\
\vdots & \vdots & \vdots & \vdots \\
1 & -1 & -1  & 1 \\
0 & -1 & s-1 & 2-s \\
2-s & s- 1 & -1 & 0
\end{pmatrix}
= \begin{pmatrix}
0 & 0 & 0  & 0 \\
* & \ominus & \ominus & * \\
\vdots & \vdots & \vdots & \vdots \\
* & \ominus & \ominus  & * \\
0 & 0 & 0 & 0 \\
 0 &  0 & 0 & 0
\end{pmatrix}
+_{sc}
\begin{pmatrix}
1 & -1 & -1  & 1 \\
\oplus & * & *  & \oplus \\
\vdots & \vdots & \vdots & \vdots \\
\oplus & * & *  & \oplus \\
 0 & -1 & s-1 & 2-s \\
2-s & s- 1 & -1 & 0
\end{pmatrix}}.
\]

The remaining $s-3$ rows of ${\bf v}$ are either $(1, -1, -1, 1)$ or $(0, 0, 0, 0)$  and considering that the sum of every column should be zero,
then we get that all rows of ${\bf v}$
should be  $(0, 0, 0, 0)$ thus  ${\bf v}$ is the zero matrix, a contradiction to our assumption that $\ub=\vb+_{sc}\wb$ is a proper semiconformal decomposition.

In conclusion, from the two cases, we obtain that ${\bf u}$ is indispensable, and therefore $m(A_s)\geq s$. \qed

\end{proof}

\section{Markov complexity of $m\times n$ matrices}

\medskip

We can state now the main result of this paper:

\begin{Theorem}
\label{final} Markov complexity of $m\times n$ matrices of rank $d$ may be arbitrarily large, for all $n\geq 4$ and $d\leq n-2$.
\end{Theorem}
\begin{proof} Let $n, d$ that satisfy the hypotheses of the theorem. In particular, we have $n-d\geq 2$. If $n-d=2$ then let
 \[ A_s= \begin{pmatrix}
 0 & 1 & s-1  & s \\
1 & 1 & 1 & 1
\end{pmatrix} \] and if $n-d\geq 3$ let
\[
A_s= \begin{pmatrix}
 1 & s & s^2-s  & s^2-1 & 1 & 1 & \cdots & 1
\end{pmatrix}
\] be an $1\times(n-d+1)$ matrix. Both of them have Markov complexity at least $s$, see Theorem~\ref{24}
for the first one and \cite[Theorem 4.1, Corollary 4.4]{KT}, respectively, for the second one.
Next we apply \cite[Theorem 2.1]{PTV} to construct $L$ a generalized Lawrence matrix $m\times n$
of rank $d$ such that its subbouquet ideal is the toric ideal of $A_s$. The matrix $L$ has rank $d$ since bouquets
preserve the codimension, as noticed in the introduction of Section 2.
 Applying now Theorem~\ref{Markovcomplexity} we obtain that the $m\times n$ matrix $L$ of rank $d$
 has Markov complexity at least $s$. \qed

\end{proof}

The next example shows the construction of $L$ from Theorem~\ref{final}.

\begin{Example} \label{Ex1}
 {\em Suppose that $d=6$ and $n=17$, then $n-d=11$ and we consider the  $1\times (n-d+1)=1\times 12$ matrix
{\small \[ A_s= \left( \begin{array}{cccccccccccc}
 1 & s & s^2-s  & s^2-1 & 1 & 1 & 1 & 1 & 1 & 1 & 1 & 1
\end{array} \right). \]}
Next we apply  \cite[Theorem 2.1]{PTV} to construct $L$, a generalized Lawrence matrix of size $m\times 17$
and rank $6$, such that its bouquet ideal is the toric ideal of $A_s$. (We should have 12 bouquets and 17-vectors.)
Take any 12 integer vectors, $(c'_{i1},\cdots , c'_{im_i})$, each of them with all of their coordinates nonzero, the first coordinate positive and the greatest common divisor of its
coordinates equal to 1, having all together (the 12 vectors) exactly 17 coordinates, for example $(1,-1), (1), (1,-1), (1), (1),$ $(1,11), (1), (1), (3, 7, 2021), (1), (1), (1)$.

We compute an arbitrary integer solution for each equation
$1=\lambda_{i1}c'_{i1}+\cdots+\lambda_{im_i}c'_{im_i}$. For example, for the vector $(1,-1)$: $1=1\cdot 1+0\cdot (-1)$, for $(1, 11)$: $1= 1\cdot 1 +0\cdot 11$
and for $ (3, 7, 2021)$: $ 1=(-2)\cdot 3+ 1\cdot 7+0\cdot 2021$. The  generalized Lawrence matrix $L$ is constructed as follows (see \cite[Example 2.4]{PTV} for another detailed example)
{\footnotesize
\[
L =
\left( \begin{array}{ccccccccccccccccc}
1 & 0 & s & s^2-s & 0 & s^2-1 & 1  & 1 & 0 & 1 & 1 & -2 & 1 & 0 & 1 & 1 & 1 \\
1 & 1 & 0 & 0 & 0 & 0 & 0 & 0 & 0 & 0 & 0 & 0 & 0 & 0 & 0 & 0 & 0 \\
0 & 0 & 0 & 1 & 1 & 0 & 0 & 0 & 0 & 0 & 0 & 0 & 0 & 0 & 0 & 0 & 0 \\
0 & 0 & 0 & 0 & 0 & 0 & 0 & -11 & 1 & 0 & 0 & 0 & 0 & 0 & 0 & 0 & 0 \\
0 & 0 & 0 & 0 & 0 & 0 & 0 & 0 & 0 & 0 & 0 & -7 & 3 & 0 & 0 & 0 & 0 \\
0 & 0 & 0 & 0 & 0 & 0 & 0 & 0 & 0 & 0 & 0 & -2021 & 0 & 3 & 0 & 0 & 0 \\
\end{array} \right) .
\]}
Finally, Theorem~\ref{Markovcomplexity} states that the $6\times 17$ constructed
matrix has Markov complexity at least $s$. }
\end{Example}

Since the minimal Markov bases of $A^{(r)}$ are inside the Graver basis of $A^{(r)}$ \cite[Section 7]{St},
then Theorem~\ref{final} leads to the following corollary. The only case missing is  the case of
$1\times 3$ whose proof is in \cite[Theorem 4.2]{HaraThomaVladoiu14}.

\begin{Corollary}
 Graver complexity of $m\times n$ matrices of rank $d$ may be arbitrarily large, for $n\geq 3$ and $d\leq n-2$.
\end{Corollary}

\begin{Example} \label{Ex2} {\em Note that \cite[Theorem 4.2]{PTV2} can be used to produce examples of $0/1$ matrices with Markov (and Graver) complexity arbitrarily high. The problem is that for these examples the size of the matrices also increases. For example, one can use the matrix
{\small
\[
A_n= \begin{pmatrix}
0 & 1 & n-1  & n \\
1 & 1 & 1 & 1
\end{pmatrix}
\]} and the results of  Section~\ref{Section: MainResult} to produce a $0/1$ matrix of size $(2n+5)\times(2n+5)$ and rank $2n+3$, whose Markov complexity is at least $n$. Concrete, for $n=5$ the matrix is $${\small\begin{pmatrix}
0 & 1 & 4  & 5 \\
1 & 1 & 1 & 1
\end{pmatrix}}$$
and the corresponding $0/1$ matrix with Markov complexity at least 5 is
{\footnotesize
\[
L =
\left( \begin{array}{ccccccccccccccc}
0 & 0 & 0 & 1 & 1 & 1 & 1 & 1 & 0 & 1 & 1 & 1 & 1 & 1 & 0  \\
0 & 0 & 1 & 0 & 1 & 1 & 1 & 1 & 0 & 1 & 1 & 1 & 1 & 1 & 0  \\
0 & 0 & 1 & 0 & 0 & 1 & 1 & 1 & 1 & 1 & 1 & 1 & 1 & 1 & 0  \\
0 & 0 & 1 & 0 & 1 & 0 & 1 & 1 & 1 & 1 & 1 & 1 & 1 & 1 & 0  \\
0 & 0 & 1 & 0 & 1 & 1 & 0 & 1 & 1 & 1 & 1 & 1 & 1 & 1 & 0  \\
0 & 0 & 1 & 0 & 1 & 1 & 1 & 0 & 1 & 1 & 1 & 1 & 1 & 1 & 0  \\
0 & 0 & 1 & 0 & 1 & 1 & 1 & 1 & 0 & 1 & 1 & 1 & 1 & 1 & 0  \\
0 & 0 & 1 & 0 & 1 & 1 & 1 & 1 & 0 & 0 & 1 & 1 & 1 & 1 & 1  \\
0 & 0 & 1 & 0 & 1 & 1 & 1 & 1 & 0 & 1 & 0 & 1 & 1 & 1 & 1  \\
0 & 0 & 1 & 0 & 1 & 1 & 1 & 1 & 0 & 1 & 1 & 0 & 1 & 1 & 1  \\
0 & 0 & 1 & 0 & 1 & 1 & 1 & 1 & 0 & 1 & 1 & 1 & 0 & 1 & 1  \\
0 & 0 & 1 & 0 & 1 & 1 & 1 & 1 & 0 & 1 & 1 & 1 & 1 & 0 & 1  \\
0 & 0 & 1 & 0 & 1 & 1 & 1 & 1 & 0 & 1 & 1 & 1 & 1 & 1 & 0  \\
1 & 0 & 1 & 0 & 1 & 0 & 0 & 0 & 0 & 1 & 0 & 0 & 0 & 0 & 0  \\
0 & 1 & 1 & 0 & 1 & 0 & 0 & 0 & 0 & 1 & 0 & 0 & 0 & 0 & 0
\end{array} \right) ,
\]}
which is a $15\times 15$ matrix of rank $13$ (see also \cite[Example 4.1]{PTV2} for details). } \end{Example}

\begin{Remark} {\em
Note that Theorem \ref{final} and Example \ref{Ex1} shows that to construct
a matrix that has Markov complexity $s$, using the techniques of this paper, the maximum absolute value of the entries of the matrix will be $s$ or $s^2-1$. Of course one can use the techniques of the paper \cite{PTV2}
 to produce examples of 0/1-matrices with Markov complexity $s$,
 like in the Example~\ref{Ex2}.
 But in this case the size of the matrix increases linearly on $s$.}
\end{Remark}

\section{Bounds on Markov complexity}

\medskip

The above results show that if we only bound the matrix dimensions then the Markov and Graver
complexities can be arbitrarily large. In contrast, we now show that if we also bound the
matrix entries, these complexities become bounded. We use recent results on sparse
integer programming from \cite{EHKKLO,KLO,KO}. We need some definitions.
The {\em height} of a rooted tree is the maximum number of vertices on a path from the
root to a leaf. Given a graph $G=(V,E)$, a rooted tree on $V$ is {\em valid} for $G$ if
for each edge $\{j,k\}\in E$ one of $j,k$ lies on the path from the root to the other.
The {\em tree-depth} $\td(G)$ of $G$ is the smallest height of a rooted tree which is valid for $G$.
For instance, if $G=([2m],E)$ is a perfect matching with $E=\{\{i,m+i\}:i\in[m]\}$
then its tree-depth is $3$ where a tree validating it rooted at $1$ has edge set
$E\uplus\{\{1,i\}:i=2,\dots,m\}$. The graph of an $m\times n$ matrix $A$ is the graph $G(A)$ on $[n]$
where ${j,k}$ is an edge if and only if there is an $i\in[m]$ such that $A_{i,j}A_{i,k}\neq 0$.
The {\em tree-depth} of $A$ is the tree-depth $\td(A):=\td(G(A))$ of its graph.

The tree-depth is a parameter playing a central role in sparsity, see \cite{NO}.
The results of \cite{EHKKLO,KLO,KO} assert that {\em sparse integer programming},
where the defining matrix $A$ has small tree-depth and hence is sparse, is efficiently solvable.
One of the key ingredients in the efficient solution is that there exists a bound on the $1$-norm
of the elements of the Graver basis of $A$ which is a computable function depending only on
$a:=\|A\|_\infty:=\max_{i,j}|A_{i,j}|$, the maximum absolute value of any entry of $A$,
and $t:=\td(A^T)$, the tree-depth of its transpose. The currently best bound is in
\cite[Lemma 28]{EHKKLO} and \cite[Lemma 11]{KPW}, as follows. 

\begin{Proposition}
\label{norm_bound}
For any integer matrix $A$,
any Graver basis element $x\in\MG(A)$ satisfies
$$\|x\|_1\ \leq\ (2a+1)^{{2^t}-1}\ .$$
\end{Proposition}

\vskip.2cm
We now prove the bound on the Markov complexity and Graver complexity.

\begin{Theorem}
\label{complexity_bound}
The Markov and Graver complexities of any integer $m\times n$ matrix $A$ satisfy 
$$m(A)\ \leq\ g(A)\ \leq\ (2a+1)^{{4^n}-1}+1\ .$$
\end{Theorem}

\begin{proof}
Since $m(A)\leq g(A)$ for any matrix, is suffices to prove the bound on $g(A)$.
Thus, we need to show that for any $r$, the type of any element $x\in\MG(A^{(r)})$, that is,
the number of nonzero rows of $x$, regarded as an $r\times n$ matrix with rows
$x^1,\dots,x^r\in\ZZ^n$, satisfies the bound. First, note that if $a=0$, that is, $A$ is the
zero $m\times n$ matrix, then $\MG(A^{(r)})$ consists precisely of those matrices $x$ which
for some $1\leq i\neq k\leq r$ and some $1\leq j\leq n$, have $x^i$ equal to the $j$th
unit vector in $\ZZ^n$ and $x^k=-x^i$, and all other rows are zero. So the type
of each such matrix is $2$ and therefore $g(A)=2$ and obeys the claimed bound.      

So assume from now on that $a\geq 1$. Removing redundant rows of $A$ does not change the
kernel of $A^{(r)}$ and hence also does not change $\MG(A^{(r)})$, so we may assume $m\leq n$.
Now, consider any $r$ and the matrix $B:=A^{(r)}$ and its transpose $B^T$. We index the rows
of $B=A^{(r)}$ naturally by $b^i_j$ for $i=1,\dots,r$ and $j=1,\dots,m$, followed by $c_k$
for $k=1,\dots,n$. Then the $b^i_j$ and $c_k$ are the vertices of the graph $G(B^T)$
of the transpose of the matrix $B$.

We define a rooted tree $T$ on the $b^i_j$ and $c_k$
as follows. We let the root be $c_1$. We include the path $c_1,c_2,\dots,c_n$.
For $i=1,\dots,r$ we include the edge $\{c_n,b^i_1\}$ and the path $b^i_1,b^i_2,\dots,b^i_m$.
We now claim that $T$ is valid for the graph $G(B^T)$. Consider any two distinct rows of $B$.
If they are some $c_k,c_l$ with $k<l$ then $c_k$ is on the path from the root to $c_l$.
If they are some $c_k$ and $b^i_j$ then $c_k$ is on the path from the root to $b^i_j$.
If they are some $b^i_k$ and $b^i_l$ with $k<l$ then $b^i_k$ is on the path from the root to $b^i_l$.
Finally, suppose they are $b^i_j$ and $b^k_l$ with $i\neq k$. Then for any column $s$ of
$B=A^{(r)}$, either $B_{b^i_j,s}=0$ or $B_{b^k_l,s}=0$, and therefore $\{b^i_j,b^k_l\}$ is
{\em not} an edge of $G(B^T)$. Therefore $T$ is valid for $G(B^T)$.

Since the height of $T$ is $n+m\leq 2n$, we find that $\td(B^T)=\td(G(B^T))\leq 2n$.
Also, since $a\geq 1$, we have that $\|B\|_\infty=a$. Therefore, by Proposition \ref{norm_bound}
applied to $B$, any Graver basis element $x\in\MG(A^{(r)})$ satisfies
$\|x\|_1\leq (2a+1)^{{2^{2n}}-1}=(2a+1)^{{4^n}-1}$. Regarding again $x\in\MG(A^{(r)})$ as an
$r\times n$ matrix with rows $x^1,\dots,x^r$, and noting that $x$ is integer matrix and therefore
$x^k\neq 0$ implies $\|x^k\|_1\geq 1$, we find that the type $|\{k:x^k\neq 0\}|$ of $x$ satisfies
$$|\{k:x^k\neq 0\}|\ \leq\ \sum_{k=1}^n \|x^k\|_1\ =\ \|x\|_1\ \leq\ (2a+1)^{{4^n}-1}\ .$$
So the type of any element $x\in\MG(A^{(r)})$ for any $r$ satisfies the claimed bound.
Since the Graver complexity $g(A)$ is the maximum such type, it satisfies the bound as well.
\end{proof}

As the Markov complexity $m(A)$ may be much smaller than the Graver complexity $g(A)$,
it is interesting to ask whether the bound on the Markov complexity can be substantially reduced,
say, to a bound which is singly, rather than doubly, exponential in $n$.

\begin{Question}
What is the best upper bound $f(a,n)$ on the Markov complexity
of any integer $m\times n$ matrix with $\|A\|_\infty=a$?
\end{Question}

Since applications of the Markov and Graver complexities are mainly for 0/1-matrices,
see \cite{AT, BerOnn, DO, DS, DSS, FinHem, HS, KudTak, OH, Onn, SS},
it is also interesting to find the best bound for $0/1$ matrices.
Note that for $0/1$ matrices $a:=\|A\|_\infty=1$ by definition.

\begin{Question}
\label{upper_ko}
What is the best upper bound $h(n)$ on the Markov complexity of any $0/1$ valued $m\times n$ matrix?
\end{Question}

We note that in \cite{TaThVl} it is proved that certain $0/1$-matrices have Markov complexity
bounded above by $n-d$. We also note that, while the bound $h(n)$ for $0/1$ matrices might be
smaller than the above bound $3^{4^n-1}$, it is known that even for $0/1$ matrices, Markov bases can
be arbitrarily complicated, and when one allows the dimensions $m$, $n$, and $r$ to grow, the
Markov bases of $A^{(r)}$ can include arbitrarily prescribed integer vectors, see \cite{DO} and \cite[Section 3]{PTV2}.

\bigskip
{\bf Acknowledgments.} The first author was supported by a grant from the Israel Science Foundation
and by the Dresner chair.
The third author has been partially supported by the grant PN-III-P4-ID-PCE-2020-0029, within PNCDI III, financed by Romanian Ministry of Research and Innovation, CNCS - UEFISCDI.

\end{document}